\documentclass[12pt]{amsart}

\usepackage[latin9]{inputenc} 
\usepackage[mathscr]{eucal}
\usepackage{amsmath,amssymb,amscd,amsthm}
\usepackage{color}
\usepackage{amsmath,amsthm,amssymb,graphicx}
\usepackage{epsfig}
\newtheorem{theorem}{Theorem}

\theoremstyle{definition}
\newtheorem{remark}{Remark}
\theoremstyle{plane}

\def \beq{ \begin{equation} }
\def \eeq{\end{equation}}

\def \R {\mathbb R}
\def \M {\mathbb M}
\def \H {\mathbb H}
\def \S {\mathbb S}

\usepackage[top=3.8cm, bottom=3.7cm, left=3.5cm, right=3.5cm]{geometry}

\title{Bifurcations of the Lagrangian orbits from the classical to the curved 3-body problem}
\begin{document}
\maketitle
\markboth{Florin Diacu}{Lagrangian orbits in the classical and curved 3-body problem}
\author{\begin{center}{\bf Florin Diacu}\\
\small Pacific Institute for the Mathematical Sciences\\ and\\
Department of Mathematics and Statistics\\
University of Victoria\\
Victoria, Canada\\
Email: diacu@uvic.ca
\end{center}}

\begin{center}
\today
\end{center}


\begin{abstract}
We consider the 3-body problem of celestial mechanics
in Euclidean, elliptic, and hyperbolic spaces, and study how the Lagrangian (equilateral) relative equilibria bifurcate when the Gaussian curvature varies. We thus prove the existence of new classes of orbits. In particular, we find some families of isosceles triangles, which occur in elliptic space.

\end{abstract}

\vspace{-0.2cm}

\section{Introduction}

The idea of extending Newtonian gravitation to spaces of constant curvature appeared for the first time in the work of J\'anos Bolyai and Nikolai Lobachevsky, who independently considered it in the framework of hyperbolic geometry, \cite{Bol}, \cite{Lob}. They apparently thought in terms of Gauss's flux law for gravity, according to which two celestial bodies attract each other with a force inversely proportional to the area of a sphere of radius equal to the distance between the bodies. 

As it happened to their work in hyperbolic geometry, it took many years until the research community recognized the importance of studying the motion of point masses gravitating in spaces of constant curvature. Lejeune Dirichlet apparently grasped the value of this problem in the 1850s, but in spite of researching it he published nothing in this direction, \cite{Sch1}. The first to write down the expression of the potential in hyperbolic space was Ernest Schering, almost four decades after Bolyai and Lobachevsky. In his derivation, Schering used the fact that the area of a sphere of radius $r$ in $\mathbb H_\kappa^3$ is $4\pi|\kappa|^{-1}\sinh^2(|\kappa|^{1/2} r)$, where $\mathbb H_\kappa^3$ denotes the hyperbolic space of curvature $\kappa<0$, \cite{Sch1}, \cite{Sch2}. It was then natural to extend this problem to elliptic space, so Wilhelm Killing defined a force inversely proportional to the area $4\pi\kappa^{-1}\sin^2(\kappa^{1/2}r)$ of a sphere of radius $r$ in the complete elliptic geometry of the sphere $\mathbb S_\kappa^3$ of curvature $\kappa>0$, \cite{Killing}. Currently, the topic is intensely researched thanks to the new approach initiated in \cite{Diacu01}, \cite{Diacu11}, and \cite{Diacu12} (see also \cite{Diacu02},
\cite{Diacu03}, \cite{Diacu05}, \cite{Diacu07}, \cite{Diacu77},
\cite{Diacu06}, \cite{Diacu08}, \cite{Diacu09}, \cite{Diacu10}, \cite{Diacu13}, \cite{Diacu-Popa}, \cite{Diacu14}, \cite{Garcia}, \cite{Mar}, \cite{Mar2}, \cite{Perez}, \cite{Shc}, \cite{Tib1}, \cite{Tib2}, \cite{Tib3}, \cite{Zhu}).

To get more insight into the problem, let's consider the motion of two point masses, $m_1$ and $m_2$, in the three existing kinds of spaces of constant curvature we are interested in, $\mathbb H_\kappa^3$ ($\kappa<0$), $\mathbb R^3$ ($\kappa=0$), and $\mathbb S_\kappa^3$ ($\kappa>0$), and assume that the corresponding hyperbolic, Euclidean, and elliptic distance between the point masses is a function $r=r(t)$. Let us take the units such that the gravitational constant is 1. Then the attracting forces are given by 
$$
F_{\mathbb H_\kappa^3}(r)=\frac{|\kappa| m_1m_2}{\sinh^2(|\kappa|^{1/2}r)},\ \ \ F_{\mathbb R^3}(r)=\frac{m_1m_2}{r^2},\ \ \ F_{\mathbb S_\kappa^3}(r)=\frac{\kappa m_1m_2}{\sin^2(\kappa^{1/2}r)}.
$$
The corresponding force functions (the negatives of the potentials, whose derivatives relative to $r$ provide the above forces), take the form
$$
V_{\mathbb H_\kappa^3}(r)=m_1m_2\coth(|\kappa|^{1/2} r),\ \ \ V_{\mathbb R^3}(r)=\frac{m_1m_2}{r},\ \ \ V_{\mathbb S_\kappa^3}(r)=m_1m_2\cot(\kappa^{1/2} r).
$$
The classical Newtonian law is recovered in the limit since
$$
\lim_{\kappa\to 0,\ \! \kappa<0}F_{\mathbb H_\kappa^3}(r)
=
\lim_{\kappa\to 0,\ \! \kappa>0}F_{\mathbb S_\kappa^3}(r)
=
F_{\mathbb R^3}(r),
$$
and similar relationships stand true for the force functions.

But in spite of recovering the Newtonian law in the limit, how
can we know that this is the most natural extension of gravity to spaces of constant curvature? After all, there are infinitely many ways of obtaining the classical force when $\kappa\to 0$. In the absence of any physical or observational tests, is Gauss's law good enough reason for the introduction of the above definitions? Some researchers obviously believed it was not, for Rudolf Lipschitz came up with another expression of the force in curved space, \cite{Lip1}. His proposed law, however, was short lived. First, the solutions to his equations of motion involved elliptic integrals, so they could not be solved. Second, some strong arguments occurred in favour of the approach initiated by Bolyai and Lobachevsky. Indeed, at the beginning of the 20th century, Heinrich Liebmann proved two important results relative to the Kepler problem (which studies the motion of one body about a fixed attractive centre). The first was that, like in the Euclidean case, the potential is a harmonic function, i.e.\ a solution of the Laplace-Beltrami equation. The second property showed that all bounded orbits are closed, a result originally proved by Joseph Bertrand for the Newtonian force in Euclidean space, \cite{Ber}. These reasons were convincing enough to accept, even though only on the basis of mathematical analogies and in the absence of physical experiments or observations, that the force functions $V_{\mathbb H^3}$ and $V_{\mathbb S^3}$ provide the correct extension of $V_{\mathbb R^3}$ to spaces of constant curvature. In fact, this conclusion should not be too surprising: Newton's gravitational law  does not shed any light on the physical nature of gravity, it only describes the motion of celestial bodies fairly well. But as we will further explain, the above generalization of the gravitational law may also have merits that transcend mathematics, although this is first of all a mathematical problem.

In 1821 Carl Friedrich Gauss performed some topographic measurements in which he measured the angles of a triangle formed by three mountain peaks, \cite{Diacu05}. His aim was apparently to find out whether space was hyperbolic or elliptic, should the sum of the angles add to less or more than $\pi$ radians. His attempt was inconclusive since the measurement errors were above the potential deviation from $\pi$. Lobachevsky tried to decide whether the physical space was curved  by measuring the parallax of Sirius, which he treated as the ideal point of an angle of parallelism, but he couldn't  draw any conclusion either. These and other 19th-century attempts to determine the curvature of space are described in \cite{Kragh}. In his famous 1854-paper that laid the foundations of differential geometry, Bernhard Riemann reiterated the importance of this problem, \cite{Riem}.  More recently the so-called boomerang experiment, involving the background radiation, was also initiated in the hope to answer this question, but again without success, \cite{Diacu05}. All these efforts proved, however, that even if the large-scale universe is not Euclidean, the deviation from zero curvature, if any, must be extremely small. 

The mathematical extension of gravitation to spaces of non-zero constant curvature offers another way to approach the problem of determining the curvature of the physical space. 
If, for instance, we could prove that certain orbits exist only in, say, Euclidean space, but not in hyperbolic and elliptic space, and we succeed to find these orbits through astronomical observations, then we would be able to conclude that the universe is flat. So the study of the $N$-body problem in spaces of constant Gaussian curvature (or curved $N$-body problem, as we informally call it), may present interest beyond mathematics.

\section{Our goal}

In the Newtonian 3-body problem of the Euclidean space there are two classes of relative equilibria, the Lagrangian and the Eulerian, named after those who proved their existence. Leonhard Euler found the collinear orbits in 1762, \cite{Euler}. Joseph Louis Lagrange rediscovered them a decade later, but also found the class of relative equilibria given by equilateral triangles, \cite{Lagrange}. In this paper we will study the latter type of solutions and analyze how they bifurcate from the Euclidean plane to 2-spheres and hyperbolic 2-spheres. The reason why we restrict our study to the 2-dimensional case is that all relative equilibria are planar in the Euclidean space. However, it is important to mention that in the curved problem there are relative equilibria that do not necessarily move on great 2-spheres or great hyperbolic 2-spheres, but this phenomenon can take place only for more than three bodies. In previous work, we provided such examples for the curved 3-dimensional problem, \cite{Diacu03}, \cite{Diacu05}, \cite{Diacu06}. So when investigating the motion of more than three bodies, a 3-dimensional study would also be necessary. In this paper, however, we can restrict our considerations to the 2-dimensional case without any loss of generality.  

It will be interesting to notice in the following sections that the dynamics on 2-spheres is richer than on hyperbolic 2-spheres, in the sense that there are more triangular relative equilibria on $\mathbb S_\kappa^2$ than on $\mathbb H_\kappa^2$. Perhaps one of the reasons for this difference can be found in the inequalities
$$
\frac{1}{\sinh^2 r}<\frac{1}{r^2}<\frac{1}{\sin^2 r},
$$
which imply that 
$$
F_{\mathbb H_\kappa^2}<F_{\mathbb R^2}<
F_{\mathbb S_\kappa^2},
$$
so the component of the force derived from the potential (since, for $\kappa\ne 0$, the acceleration also involves the force due to the constraints, which keep the bodies on the manifolds) is stronger on spheres than
on hyperbolic spheres.

\section{Summary of the results}

After introducing the equations of motion in Section 4, we study in Section 5 the existence of relative equilibria on and near the equator of $\mathbb S_\kappa^2$. Our first theorem
provides a different proof for a result we first published in \cite{Diacu02}, namely that for every acute scalene triangle inscribed in the equator, we can find a class of masses
$m_1, m_2, m_3>0$, which if placed at the vertices of the triangle form a relative equilibrium that rotates around the equator with any chosen nonzero angular velocity. Then in Theorem 2 we prove a qualitative property: if the three bodies move either in the northern or in the southern closed hemisphere and one of the bodies is on the equator, then all three bodies must move on the equator. In Theorem 3, we find new classes of relative equilibria that move on non-great circles parallel with the plane of the equator, namely those given by isosceles non-equilateral triangles. These relative equilibria occur for masses $m_1=:M>0$ and $m_2=m_3=:m>0$, with $M<2m$, in two pairs of bands symmetric to the equator, as shown in Figure 2. For one of the classes the shape of the isosceles triangle is unique for the given masses, whereas in the other class two distinct shapes are possible. For $M=m$ we recover the Lagrangian (equilateral) relative equilibria, which exist on all circles parallel with (and including) the equator, a result we first proved in \cite{Diacu11}.

In Section 6 we study the existence of relative equilibria parallel with the $xy$-plane in $\mathbb H_\kappa^2$.
We witness here the first manifestation of the difference
between the richness of orbits that occur on spheres and hyperbolic spheres by proving that there are no isosceles relative equilibria parallel with the $xy$-plane other than the Lagrangian solutions of equal masses. In Section 7 we introduce some equivalent form of the equations of motion that is more suitable for the study of Lagrangian (equilateral) relative equilibria. This form of the equations has been suggested to us by Carles Sim\'o, who used them in a recent paper on the restricted curved 3-body problem he wrote with 
Regina Mart\'inez, \cite{Mar2}. In Section 8 we take a glimpse at the simple case of the planetary problem, in which two masses are negligible, and prove in Theorem 5 that there occur no bifurcations of the Lagrangian relative equilibria when passing from $\mathbb S_\kappa^2$ to $\mathbb R^2$ to $\mathbb H_\kappa^2$ as kappa goes from $+\infty$ to $-\infty$.

In Section 9 we focus on the case of one negligible mass.
In Theorem 6 we assume that two bodies of equal mass move on a non-equatorial circle of the sphere $\mathbb S_\kappa^2$, being always diametrically opposed, and form a Lagrangian relative equilibrium with a third body, which has negligible mass. Then the circle on which the two bodies move must have its radius equal to $(2\kappa)^{-1/2}$ and the third body must move on the equator. In other words, given $\kappa$, the size of the equilateral triangle does not depend on the value of the equal masses, but the angular velocity of the equilateral triangle does. In Theorem 7 we then show
that there are no Lagrangian relative equilibria in $\mathbb H_\kappa^2$ with two bodies of equal mass and a third body of negligible mass, finding again a manifestation of the difference between the richness of orbits that occur on spheres and hyperbolic spheres. Finally in Theorem 8 we show that
if one of the three masses is negligible, then
there are no Lagrangian relative equilibria in $\mathbb H_\kappa^2$ and there are no Lagrangian relative equilibria in $\mathbb S_\kappa^2$ either if the curvature $\kappa$ is sufficiently small, unless the two non-negligible masses are equal, in which case the orbits occur as stated in Theorem 6.

\section{Equations of motion}

The goal of this section is to define the spaces of constant curvature in which the bodies move and introduce the equations of motion that extend Newton's classical system beyond the Euclidean case. Consider for this the family of 2-dimensional manifolds $(\M_\kappa^2)_{\kappa\in\R}$, with
$$
\M_\kappa=
\begin{cases}
\S_\kappa^2\hspace{0.3cm} {\rm for}\ \ \kappa>0\cr
\R^2\hspace{0.25cm} {\rm for}\ \ \kappa=0\cr
\H_\kappa^2\hspace{0.21cm} {\rm for}\ \ \kappa<0,
\end{cases}
$$
where the set $\R^2$ is the horizontal Euclidean plane of curvature $\kappa=0$ through the origin of the coordinate system,
$$
\R^2=\{(x,y,z)\ \!|\ \! z=0\},
$$ 
the sets $\S_\kappa^2$ denote the spheres
$$
\mathbb S_\kappa^2=\{(x,y,z)\ \!|\ \! \kappa(x^2+y^2+z^2)+2\kappa^{1/2}z=0\}
$$
centred at $(0,0,-\kappa^{-1/2})$ of curvature $\kappa>0$,
and the sets $\H_\kappa^2$ are the hyperbolic spheres of curvature $\kappa<0$ represented by the upper sheets of hyperboloids of two sheets,
$$
\H_\kappa^2=\{(x,y,z)\ \!|\ \! \kappa(x^2+y^2-z^2)+2|\kappa|^{1/2}z=0,\ z\ge 0\},
$$
whose vertex is tangent to the $xy$-plane.
The spheres $\S_\kappa^2$ and the plane $\R^2$ are embedded in $\R^3$, which has the standard inner product of signature $(+,+,+)$, whereas $\H_\kappa^3$ is embedded in the Minkowski space $\R^{2,1}$, endowed with the Lorentz inner product of signature $(+,+,-)$.
All these manifolds have a single point in common, the origin $(0,0,0)$ of the coordinate system (see Figure \ref{common}).

\begin{figure}[htbp] 
   \centering
   \includegraphics[width=2in]{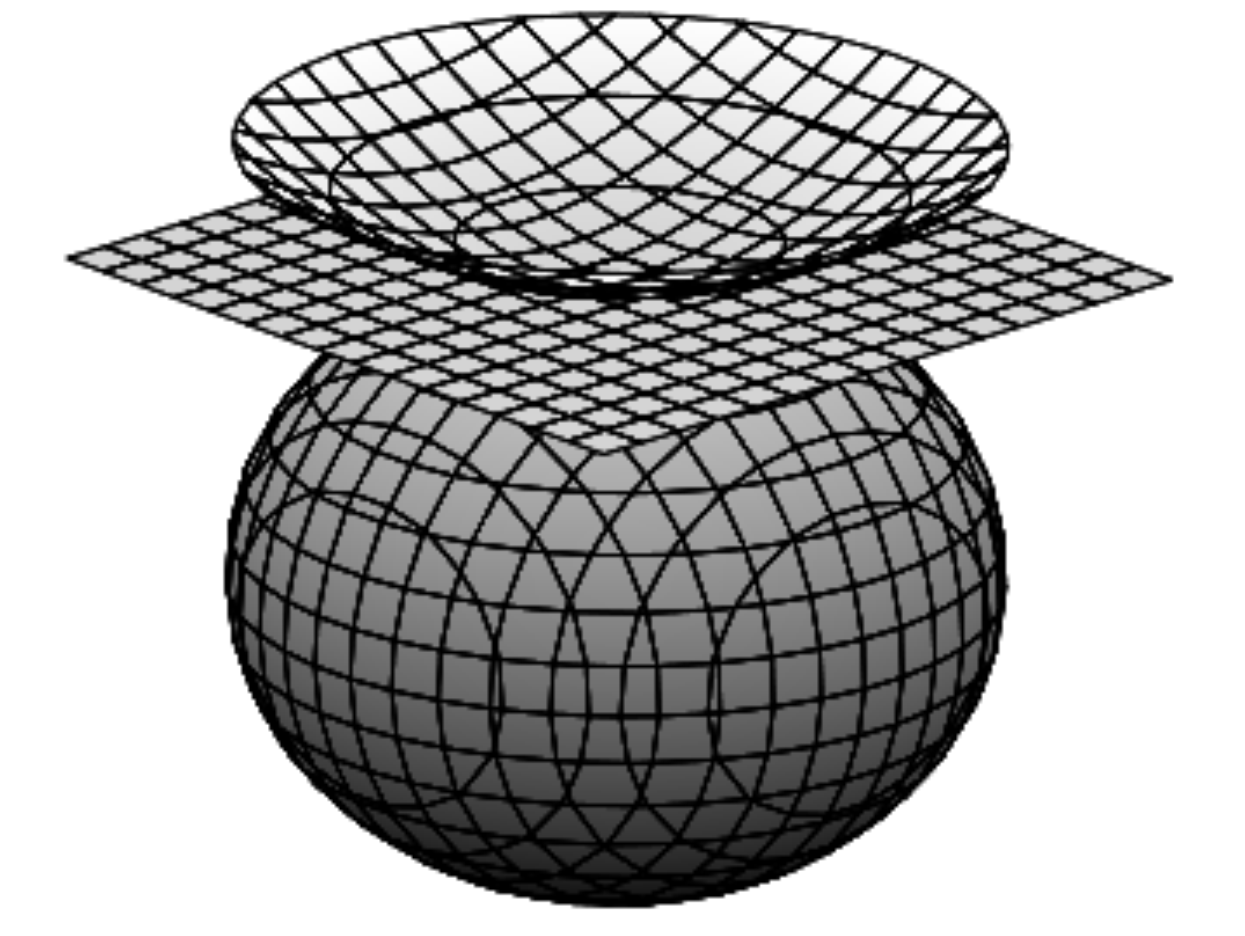}
   \caption{\small A snapshot of the continuous transition from the spheres $\mathbb S_\kappa^2$, of constant curvature $\kappa>0$, to the plane $\mathbb R^2$, of curvature $\kappa=0$, and to the hyperbolic spheres $\mathbb H_\kappa^2$, of curvature $\kappa<0$, as $\kappa$ decreases from $+\infty$ to $-\infty$.}
   \label{common}
\end{figure}

In a previous paper, \cite{Diacu77}, we obtained the equations of motion of the $N$-body problem on the above 2-dimensional manifolds of constant curvature as well as in their 3-dimensional counterparts. But here we will consider only the case $N=3$ on $\mathbb M_\kappa^2$. Then the equations of motion have the form
\begin{equation}\label{new}
\begin{cases}
\ddot x_i=\sum_{j=1, j\ne i}^3\frac{m_j\Big[x_j-\Big(1-\frac{\kappa r_{ij}^2}{2}\Big)x_i\Big]}{r_{ij}^3\Big(1-\frac{\kappa r_{ij}^2}{4}\Big)^{3/2}}-\kappa (\dot{\bf r}_i\cdot\dot{\bf r}_i)x_i\cr
\ddot y_i=\sum_{j=1, j\ne i}^3\frac{m_j\Big[y_j-\Big(1-\frac{\kappa r_{ij}^2}{2}\Big)y_i\Big]}{r_{ij}^3\Big(1-\frac{\kappa r_{ij}^2}{4}\Big)^{3/2}}-\kappa (\dot{\bf r}_i\cdot\dot{\bf r}_i)y_i\cr
\ddot z_i=\sum_{j=1, j\ne i}^3\frac{m_j\Big[z_j-\Big(1-\frac{\kappa r_{ij}^2}{2}\Big)z_i+\frac{\sigma|\kappa|^{1/2} r_{ij}^2}{2}\Big]}{r_{ij}^3\Big(1-\frac{\kappa r_{ij}^2}{4}\Big)^{3/2}}- (\dot{\bf r}_i\cdot\dot{\bf r}_i)(\kappa z_i+\sigma|\kappa|^{1/2}),
\end{cases}
\end{equation}
$i=1,2,3$, where $m_1,m_2,m_3>0$ represent the masses,
$$
{\bf r}_i=(x_i,y_i,z_i),\ \ \dot{\bf r}_i=(\dot{x}_i, \dot{y}_i, \dot{z}_i), \ i=1,2,3,
$$ 
are, respectively, the position vectors and the velocities of the bodies, $\sigma$ denotes the signum function: $\sigma=+1$ for $\kappa\ge 0$ and $\sigma=-1$ for $\kappa<0$, and 
$$
r_{ij}:=
[(x_i-x_j)^2+(y_i-y_j)^2+\sigma(z_i-z_j)^2]^{1/2},\ \ i,j=1,2,3,
$$
is the Euclidean distance between the bodies of masses $m_i$ and $m_j$ in $\mathbb R^3$, but the Minkowski distance in
$\mathbb R^{2,1}$.

At $t=0$, the initial conditions must have the six constraints
\begin{equation}
\label{constraint-1}
\kappa(x_i^2+y_i^2+\sigma z_i^2)+2|\kappa|^{1/2}z_i=0,\ \ i=1,2,3,
\end{equation}
\begin{equation}
\label{constraint-2}
\kappa{\bf r}_i\cdot\dot{\bf r}_i+|\kappa|^{1/2}\dot z_i=0,\ \ i=1,2,3.
\end{equation}
Since the sets $\mathbb S_\kappa^2$ and $\mathbb H_\kappa^2$ are invariant for the above equations of motion, these conditions are satisfied for all $t$. They are identically satisfied on $\mathbb R^2$, i.e.\ for $\kappa=0$.

Let us take a value of $\kappa$ and fix the point masses $m_i$ and $m_j$ on the manifold $\mathbb M_\kappa^2$. If we let $\kappa$ vary and keep $r_{ij}$ constant, then the coordinates of the point masses vary with $\kappa$. In particular, the values of $z_i,\ i=1,2,3$, and consequently the values of the expressions $(z_i-z_j)^2, \ i,j=1,2,3, i\ne j,$ become small when $\kappa$ gets close to 0 and vanish at $\kappa=0$. Consequently, for $\kappa=0$  we recover the classical Newtonian equations of the $3$-body problem in the Euclidean plane,
\begin{equation}
\ddot{\bf r}_i=\sum_{j=1, j\ne i}^3\frac{m_j({\bf r}_j-{\bf r}_i)}{r_{ij}^3}, \ \ i=1,2,3,
\end{equation}
where ${\bf r}_i=(x_i,y_i,0),\ i=1,2,3$.

\section{Relative equilibria on and near the equator of $\mathbb S_\kappa^2$}

In this section we will introduce some coordinates that allow us to better detect relative equilibria on and near the equator of $\mathbb S_\kappa^2$, namely $(\varphi,\omega)$, where $\varphi$ measures the angle from the $x$-axis in the $xy$-plane and $\omega$ is the height on the vertical $z$-axis. To express the coordinates of each body in this way, let us remark that from 
the constraints \eqref{constraint-1}, which can be written on
$\mathbb S_\kappa^2$ as
$$
x_i^2+y_i^2+\omega_i^2+2\kappa^{-1/2}\omega_i=0, \ i=1,2,3, 
$$
we obtain the relations
$$
\Omega_i:=x_i^2+y_i^2=-\kappa^{-1/2}\omega_i(\kappa^{1/2}\omega_i+2)\ge 0,\ \ i=1,2,3.
$$
Notice that in the inequality, which follows from the fact that $\omega_i\in[-2\kappa^{-1/2}, 0]$, equality occurs only when the body is at the North or South Pole.

We can now express the positions of the bodies in $(\varphi, \omega)$-coordinates with the help of the polar transformations
$$
x_i=\Omega_i^{1/2}\cos\varphi_i,\ \ y_i=\Omega_i^{1/2}\sin\varphi_i, \ \ i=1,2,3.
$$
Some straightforward computations show that the equations of motion \eqref{new} take the form
\begin{equation}\label{equatorial}
\begin{cases}
\ddot\varphi_i=\Omega_i^{-1/2}\sum_{j=1, j\ne i}^3
\frac{m_j\Omega_j^{1/2}\sin(\varphi_j-\varphi_i)}{\rho_{ij}^3\big(1-\frac{\kappa\rho_{ij}^2}{4}\big)^{3/2}}-\frac{\dot\varphi_i\dot\Omega_i}{\Omega_i}  \cr
\ddot\omega_i=\sum_{j=1, j\ne i}^3\frac{m_j\big[\omega_j-\omega_i+\frac{\kappa\rho_{ij}^2}{2}(\omega_i+\kappa^{-1/2})\big]}{\rho_{ij}^3\big(1-\frac{\kappa\rho_{ij}^2}{4}\big)^{3/2}}-(\kappa\omega_i+\kappa^\frac{1}{2})\big(\frac{\dot{\Omega}_i^2}{4\Omega_i}+\dot{\varphi}_i^2\Omega_i+\dot{\omega}_i^2\big),
\end{cases}
\end{equation}
$i=1,2,3$, where
$$
\dot\Omega_i=-2\kappa^{-1/2}\dot\omega_i(\kappa^{1/2}\omega_i+1),\ \ i=1,2,3,
$$
$$
\rho_{ij}^2=\Omega_i+\Omega_j-2\Omega_i^{1/2}\Omega_j^{1/2}\cos(\varphi_i-\varphi_j)+(\omega_i-\omega_j)^2,\ i, j=1,2,3, \ i\ne j. 
$$

\subsection{Relative equilibria on the equator} Let us first seek relative equilibria on the equator $\omega=-\kappa^{-1/2}.$ Then
$$
\omega_i=-\kappa^{-1/2},\ \ \dot\omega_i=0,\ \ \Omega_i=\kappa^{-1},\ \ \dot\Omega_i=0,\ \ i=1,2,3,
$$
$$
\rho_{ij}^2=2\kappa^{-1}[1-\cos(\varphi_i-\varphi_j)], \ i, j=1,2,3, \ i\ne j. 
$$
In this case, the equations in \eqref{equatorial} corresponding to $\ddot\omega_i, \ i=1,2,3$, are identically satisfied, and the equations corresponding to $\ddot\varphi_i, \ i=1,2,3$, lead to
the system
\begin{equation}\label{varphis}
\ddot\varphi_i=\kappa^{3/2}\sum_{j=1,j\ne i}^3\frac{m_j\sin(\varphi_j-\varphi_i)}{|\sin(\varphi_j-\varphi_i)|^3}, \ \ i=1,2,3.
\end{equation}
But for relative equilibria the angular velocity is the same constant for all particles, so if we denote this velocity by $\alpha\ne 0$  we can assume that
\begin{equation}\label{vars}
\varphi_1=\alpha t+a_1, \ \varphi_2=\alpha t+a_2,\ \varphi_3=\alpha t+a_3, 
\end{equation}
where $t$ represents the time and $a_1, a_2, a_3$ are real constants. Therefore
$$
\ddot\varphi_i=0, \ \ i=1,2,3.
$$
With the notation
$$
s_1:=\frac{\kappa^{3/2}\sin(\varphi_1-\varphi_2)}{|\sin(\varphi_1-\varphi_2)|^3}, \ \  s_2:=\frac{\kappa^{3/2}\sin(\varphi_2-\varphi_3)}{|\sin(\varphi_2-\varphi_3)|^3}, \ \ s_3:=\frac{\kappa^{3/2}\sin(\varphi_3-\varphi_1)}{|\sin(\varphi_3-\varphi_1)|^3},
$$ 
which are constants, equations \eqref{varphis} take the form
\begin{equation}
\begin{cases}
\ \ \ \! m_1s_1-m_3s_2\ \ \ \ \ \ \ \ \ \ \ =0\cr
-m_2 s_1\ \ \ \ \ \ \ \ \ \ \      +m_3s_3 \hspace{0.3mm} =0\cr
\ \ \ \ \ \ \ \ \ \ \ \ \ \ \!  m_2s_2-m_1s_3 \hspace{0.2mm} =  0.
\end{cases}
\end{equation}
This system has infinitely many solutions,
$$
s_1=\frac{m_3}{m_2}\gamma,\ \ s_2=\frac{m_1}{m_2}\gamma,\ \ s_3=\gamma,
$$
with $\gamma\ne 0$, such that $s_1, s_2, s_3$ make sense.
We have thus obtained a new proof for a result we previously published in \cite{Diacu02}:

\begin{theorem}
For every acute scalene triangle inscribed in the equator of $\mathbb S_\kappa^2$, we can find a class of masses $m_1, m_2, m_3>0$, which if placed at the vertices of the triangle form a relative equilibrium that rotates around the equator with any chosen nonzero angular velocity.
\end{theorem} 

It is interesting to note that for relative equilibria on $\mathbb S_\kappa^2$ if one body moves on the equator, then all
bodies must move on the equator, as long as the bodies are assumed to move only in the upper, or only in the lower, closed hemisphere (i.e.\ including the equator). Let us now formally state and prove this result.

\begin{theorem}
Consider a relative equilibrium on $\mathbb S_\kappa^2$ for which all the bodies move either in the northern or in the southern closed hemisphere. Then if one of the bodies moves on the equator, all the bodies move on the equator.
\end{theorem}
\begin{proof}
Assume, without loss of generality, that the bodies are in
the upper closed hemisphere, i.e. 
$$
\omega_1=-\kappa^{-1/2},\ \ \omega_2=-u,\ \ \omega_2=-v,\ \
u,v\in[0,\kappa^{-1/2}].
$$
Then the equation in \eqref{equatorial} corresponding to $\ddot\omega_1$ reduces to
$$
\frac{m_2(\kappa^{-1/2}-u)}{\rho_{12}^3\Big(1-\frac{\kappa\rho_{12}^2}{4}\Big)^{3/2}}+\frac{m_3(\kappa^{-1/2}-v)}{\rho_{13}^3\Big(1-\frac{\kappa\rho_{13}^2}{4}\Big)^{3/2}}=0.
$$
Since the masses and denominators in the left hand side of the above equation are positive and $0\le u,v\le \kappa^{-1/2}$, it follows that this equation can be satisfied only if $u=v=\kappa^{-1/2}$. Consequently all the bodies move on the equator.
\end{proof}
\begin{remark}
The generalization of the above statement and proof to any number $N\ge 3$ of bodies is straightforward.
\end{remark}

\subsection{Relative equilibria parallel with the equator}

In this subsection we will prove the existence of some
isosceles relative equilibria that rotate on non-geodesic 
circles parallel with the plane of the equator. Here is the precise statement of our result.

\begin{theorem}\label{isos-re}
For any sphere $\mathbb S_\kappa^2$ and masses $m_1=:M>0, m_2=m_3=:m>0$, with $M<2m$, there exist two non-geodesic circles parallel with the plane of the equator, symmetrically placed at distance $r\kappa^{-1/2}$ from it, such that the three bodies can form isosceles, non-equilateral, relative equilibria that rotate on any of those parallel circles. Moreover (see Figure \ref{zones}),

(i) the shape of the triangle is unique if $r\in(0,\sqrt{3}/3]\cup\{3/5\}$,

(ii) there are two possible shapes of the triangle if $r\in(\sqrt{3}/3,3/5)$.

\noindent In each case the rotation takes place with constant nonzero angular velocity, whose value depends on $\kappa, m$, and $M$.
\end{theorem}
\begin{figure}[htbp]\label{zones} 
   \centering
   \includegraphics[width=2in]{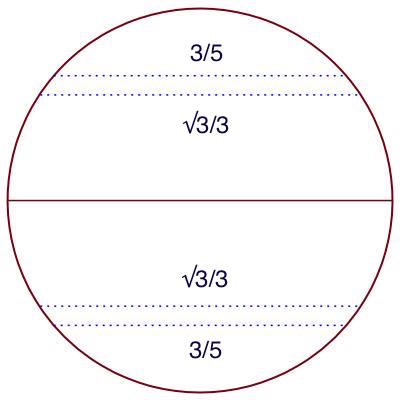}
   \caption{\small 
The zones in which one or two isosceles relative equilibria show up. The former orbits occur on the parallels marked by $3/5$ and in the large bands between the equator and the parallels marked by $\sqrt{3}/3$, whereas the latter show up in the two narrow bands, each between the parallel lines marked by $\sqrt{3}/3$ and $3/5$. The represented numbers, multiplied by $\kappa^{-1/2}$, which is the length of the sphere's radius, give the distances of the dotted lines from the plane of the equator.}
\end{figure}
\begin{proof}
We will start by seeking relative equilibria on non-great circles parallel with the equator. For this, we take
\begin{equation}\label{omega_123}
\omega_i=-u\  ({\rm constant}), \  i=1,2,3, \ \ 0<u<2\kappa^{-1/2}, \ u\ne \kappa^{-1/2}.
\end{equation}
Then we have
$$
\Omega_i=u(2\kappa^{-1/2}-u),\ \ \dot\omega_i=\dot\Omega_i=0, \ \ i=1,2,3,
$$
and assuming that the $\varphi$-angular positions are
given by the expressions in \eqref{vars}, we obtain that
$$
\rho_{ij}^2=2u(2\kappa^{-1/2}-u)[1-\cos(a_j-a_i)],\ \ i,j=1,2,3,\ i\ne j.
$$
The equations corresponding to $\ddot\omega_i, \ i=1,2,3,$ in system \eqref{equatorial} reduce to the algebraic equations
$$
\frac{\kappa(\kappa^{-1/2}-u)}{2}\Bigg[\sum_{j=1,j\ne i}\frac{m_j}{\rho_{ij}\big(1-\frac{\kappa\rho_{ij}^2}{4}\big)^{3/2}}  -2\alpha^2u(2\kappa^{-1/2}-u)\Bigg]=0, \ \ i=1,2,3.
$$
Since $\kappa>0$, we are not on the equator, i.e.\ $u\ne \kappa^{-1/2}$, and $\alpha^2$ must be the same in all the above three equations, this system leads to the conclusion that 
\begin{equation}\label{3eq-omega}
\frac{m_2}{A_{21}^{1/2}B_{21}^{3/2}}+\frac{m_3}{A_{31}^{1/2}B_{31}^{3/2}}=\frac{m_1}{A_{12}^{1/2}B_{12}^{3/2}}+\frac{m_3}{A_{32}^{1/2}B_{32}^{3/2}}=\frac{m_1}{A_{13}^{1/2}B_{13}^{3/2}}+\frac{m_2}{A_{23}^{1/2}B_{23}^{3/2}},
\end{equation}
where
$$
A_{ij}=1-\cos(a_i-a_j),\ \ B_{ij}=2-\gamma A_{ij},\ \ \gamma= \kappa u(2\kappa^{-1/2}-u).
$$
Notice that $A_{ij}=A_{ji}$ and $B_{ij}=B_{ji}, \ i,j=1,2,3, \ i\ne j$.

To evaluate the range in which $\gamma$ lies, we denote
$$
u=\lambda\kappa^{-1/2}, \ \ {\rm with}\ \  0<\lambda<2, \ \lambda\ne 1,
$$
in agreement with the conditions imposed on $u$ in \eqref{omega_123}. Then
$$
\gamma=\lambda(2-\lambda),
$$
which implies that, for the range of $\lambda$ specified above, we have
$$
0<\gamma<1.
$$

The equations corresponding to $\ddot\varphi_i,\ i=1,2,3,$ in system \eqref{equatorial} take the form
$$
\sum_{j=1,j\ne i}^3\frac{m_j\sin(a_j-a_i)}{\rho_{ij}^3\big(1-\frac{\kappa\rho_{ij}^2}{4}\big)^{3/2}}=0, \ \ i=1,2,3,
$$
which are equivalent to
\begin{equation}\label{3eq-varphi}
\begin{cases}
\cfrac{m_2\sin(a_2-a_1)}{A_{21}^{3/2}B_{21}^{3/2}}+\cfrac{m_3\sin(a_3-a_1)}{A_{31}^{3/2}B_{31}^{3/2}}=0\cr
\cfrac{m_1\sin(a_1-a_2)}{A_{12}^{3/2}B_{12}^{3/2}}+\cfrac{m_3\sin(a_3-a_2)}{A_{32}^{3/2}B_{32}^{3/2}}=0\cr
\cfrac{m_1\sin(a_1-a_3)}{A_{13}^{3/2}B_{13}^{3/2}}+\cfrac{m_2\sin(a_2-a_3)}{A_{23}^{3/2}B_{23}^{3/2}}=0.
\end{cases}
\end{equation}


It was shown in \cite{Diacu11} that, for any given $u$ as above, there exist two values for the angular velocity $\alpha$, one corresponding to each sense of rotation, in the case when the triangle is equilateral and $m_1=m_2=m_3$. It is easy to see that we can recover these relative equilibria from equations \eqref{3eq-omega} and \eqref{3eq-varphi}. We will therefore look now for acute isosceles relative equilibria. Triangles with an obtuse angle cannot form relative equilibria since it is impossible to have relative equilibria if, at every moment in time, there is a plane containing the rotation axis such that all the bodies are on one side of it (see \cite{Diacu11} for a proof of this fact). For this purpose, we can assume without loss of generality that
$$
a_1=0,\ \ a_2=:a,\ \ a_3=:2\pi-a, \ \ {\rm with}\ \ a\in(\pi/2,\pi),
$$ 
such that the isosceles triangle is acute. Then $A_{12}=A_{13}$ and $B_{12}=B_{13}$, so from the last equation in
\eqref{3eq-omega} we can draw the conclusion that $m_2=m_3$. Let us further use the notation
$$
M:= m_1,\ \ m:=m_2=m_3.
$$
Then equations \eqref{3eq-omega}-\eqref{3eq-varphi}
reduce to
\begin{equation}\label{3eq-omega2}
\frac{2m-M}{A^{1/2}B^{3/2}}=
\frac{m}{C^{1/2}D^{3/2}},
\end{equation}
\begin{equation}\label{3eq-varphi2}
\frac{M}{A^{3/2}B^{3/2}}=-\frac{2m\cos a}{C^{3/2}D^{3/2}},
\end{equation}
respectively, where
$$
A=1-\cos a,\ \ B=2-\gamma(1-\cos a),\ \ C=1-\cos 2a,\ \ D=2-\gamma(1-\cos 2a).
$$
Expressing $A^{1/2}B^{3/2}$ in \eqref{3eq-omega2} in terms of $C^{1/2}D^{3/2}$ and substituting in \eqref{3eq-varphi2},
we are led to the conclusion that 
$$
\cos a=-\frac{M}{2m}.
$$
Since $a\in(\pi/2,\pi)$, it means that
$$
-1<-\frac{M}{2m}<0,
$$
conditions that are satisfied for positive masses whenever
$$
M<2m.
$$
Notice that for $m=M$, equations \eqref{3eq-omega2}-\eqref{3eq-varphi2} are identically satisfied, so we recover the Lagrangian equilateral relative equilibria.

Substituting $M=-2m\cos a$ in equations \eqref{3eq-omega2} and \eqref{3eq-varphi2}, we are led to the same relationship, namely
$$
2-\gamma(1-\cos a)=4(1+\cos a)[1-\gamma(1-\cos^2 a)].
$$
Using the notation $s:=\cos a$, the above equation takes the form
$$
4\gamma s^3+4\gamma s^2+(4-5\gamma) s-3\gamma +2=0,
$$
which can be written as 
$$
\gamma=F(s), \ \ {\rm with}\ \ F(s)=-\frac{4s+2}{4s^3+4s^2-5s-3}.
$$
It is easy to see that for $s\in(-1,-1/2)\cup(-1/2,0)$, $F$ takes values in the interval $[16/25, 1)$, with its single minimum, $16/25$, occurring for $s=-1/4$ (see Figure \ref{graph1}). 
For $s=-1/2$, a case that corresponds to $a=2\pi/3$, i.e. to Lagrangian equilateral triangles, we have that $F(-1/2)=0/0$. This means $F$ could take any value at $s=-1/2$. The physical interpretation of this fact is that Lagrangian relative equilibria can occur on any parallel circle on the sphere.

\begin{figure}[htbp]\label{graph1} 
   \centering
   \includegraphics[width=1.5in]{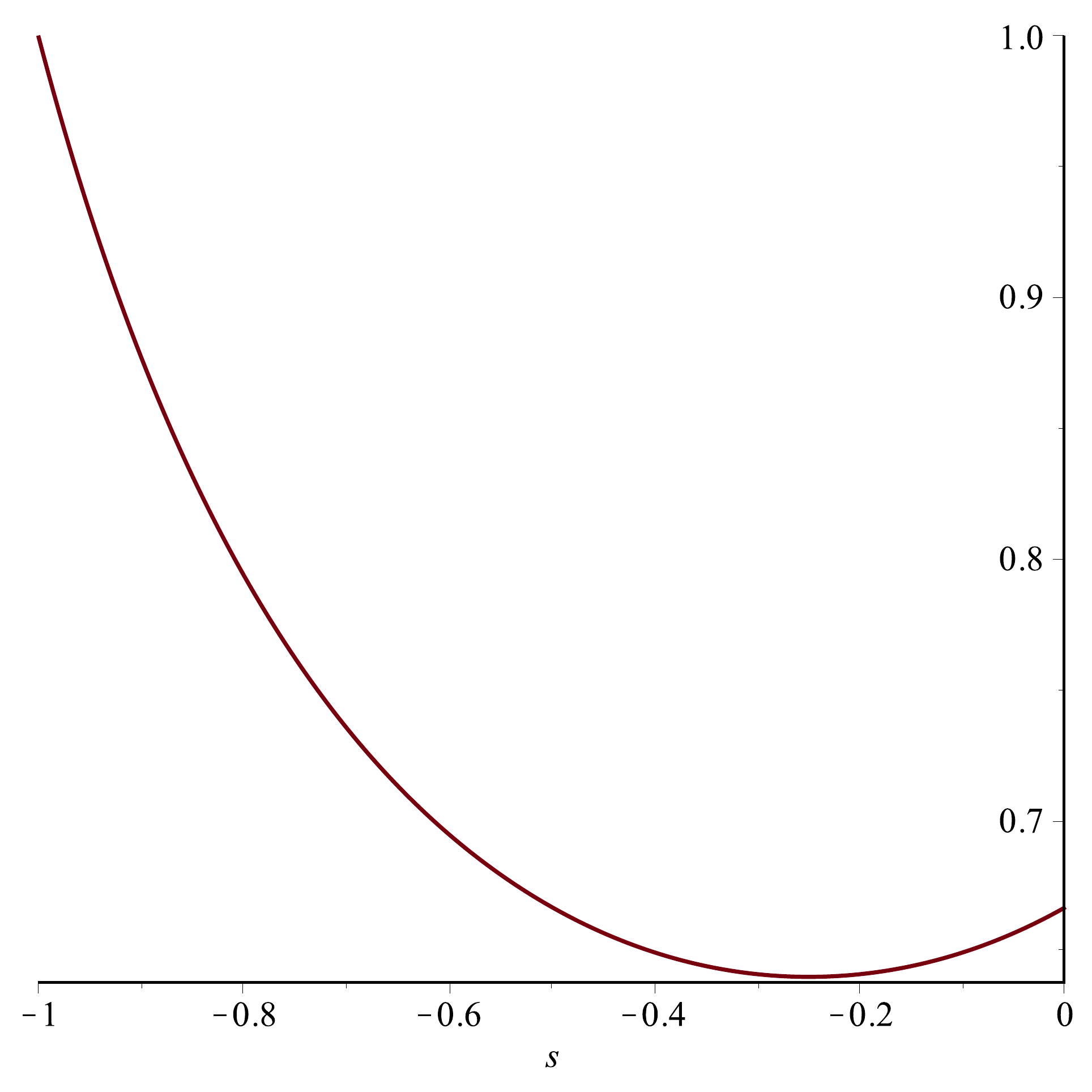}
   \caption{\small The graph of
   $F(s)=-\frac{4s+2}{4s^3+4s^2-5s-3}$ in the interval $(-1,0)$.}
   \label{graph1}
\end{figure}

For $s\in(-1,-1/2)\cup(-1/2,0)$, we necessarily have $\gamma\in[16/25,1)$. But as $\lambda(2-\lambda)=\gamma$, it means
that $\lambda\in(2/5,8/5)$. Therefore isosceles non-equilateral relative equilibria can exists only on non-geodesic circles parallel with the equator in a region bounded by two planes: one at distance $\frac{3}{5}\kappa^{-1/2}$ above the plane of the equator and the other at the same distance below the plane of the equator. Since $F(0)=2/3$ (see Figure \ref{graph1}), for $\gamma\in(16/25,2/3)$ there are two values of $s$ that satisfy the equation $\gamma=F(s)$. They correspond to values of $\lambda$ that satisfy the inequalities
$
16/25<\lambda(2-\lambda)<2/3,
$
which translate into $\lambda\in (2/5,1-\sqrt{3}/3)\cup(1+\sqrt{3}/3,8/5)$. This means that there are two open regions, symmetrically placed relative to the equator, in which for every admissible $m$ and $M$ we find two distinct isosceles triangles that form relative equilibria. These regions are distanced at $\frac{3}{5}\kappa^{-1/2}$ (upper bound) and
$\frac{\sqrt{3}}{3}\kappa^{-1/2}$ (lower bound) from the plane of the equator. When $\gamma=2/3$, there are two values of $s$ that correspond to it (see Figure \ref{graph1}), one of which is 0, and implies that $\cos a =0$, a case that leads to a right isosceles triangle (which cannot be a relative equilibrium since the triangle is not acute). When $\gamma=1$, only the value $s=-1$ corresponds to it, which means that $\cos a=-1$, so we have a degenerate isosceles triangle with a collision-antipodal singularity (collision between $m_2$ and $m_3$ and antipodal configuration between $m_1$ and the pair $m_2m_3$). Therefore this case leads to no new solutions, a remark that completes the proof.
\end{proof}

\begin{remark}
In a previous paper, \cite{Diacu02}, we stated a result according to which there are no relative equilibria given by scalene triangles on parallel circles outside the equator of $\mathbb S_\kappa^2$. But in the mean time we found an error in the proof (we wrongly assumed that a certain relation can generate two similar relations by circular permutations). The result in \cite{Diacu02} remains true as stated (and we will provide a correct demonstration in a future paper), in the sense of purely scalene triangular relative equilibria, i.e.\ with the exclusion of the isosceles relative equilibria whose existence is proved above.
\end{remark}

\section{Elliptic relative equilibria in $\mathbb H_\kappa^2$}

In this section we study the existence of relative equilibria
in $\mathbb H_\kappa^2$. It turns out that the set of solutions of this kind is not as rich as in $\mathbb S_\kappa^2$, a first manifestation of the phenomenon we mentioned earlier. More precisely we will prove the following result.  
\begin{theorem}
There are no isosceles relative equilibria parallel with the $xy$-plane in $\mathbb H_\kappa^2$, except for the Lagrangian (equilateral) solutions of equal masses, which occur on any circle parallel with the $xy$-plane.
\end{theorem}
\begin{proof}
We start by rewriting system \eqref{new} in a convenient way, which will allow us to use part of the proof of Theorem 3 for our current purposes. For this, let us denote 
$$
\Psi_i=|\kappa|^{-1/2}\omega_i(|\kappa|^{1/2}\omega_i+2),\ \ \omega_i\in[0,\infty), \ \ \kappa<0,
$$
and consider the change of coordinates
$$
x_i=\Psi_i^{1/2}\cos\varphi_i,\ \ y_i=\Psi_i^{1/2}\sin\varphi_i, \ \ i=1,2,3.
$$
Then some straightforward computations show that the equations of motion \eqref{new} take the form
\begin{equation}\label{equatorial-H}
\begin{cases}
\ddot\varphi_i=\Psi_i^{-1/2}\sum_{\stackrel{j=1}{j\ne i}}^3
\frac{m_j\Psi_j^{1/2}\sin(\varphi_j-\varphi_i)}{\rho_{ij}^3\big(1-\frac{\kappa\rho_{ij}^2}{4}\big)^{3/2}}-\frac{\dot\varphi_i\dot\Psi_i}{\Psi_i}  \cr
\ddot\omega_i=\sum_{\stackrel{j=1}{j\ne i}}^3\frac{m_j\big[\omega_j-\omega_i+\frac{\kappa\rho_{ij}^2}{2}(\omega_i+|\kappa|^{-1/2})\big]}{\rho_{ij}^3\big(1-\frac{\kappa\rho_{ij}^2}{4}\big)^{3/2}}-(\kappa\omega_i-|\kappa|^\frac{1}{2})\big(\frac{\dot{\Psi}_i^2}{4\Psi_i}+\dot{\varphi}_i^2\Psi_i-\dot{\omega}_i^2\big),
\end{cases}
\end{equation}
$i=1,2,3$, where
$$
\dot\Psi_i=2|\kappa|^{-1/2}\dot\omega_i(|\kappa|^{1/2}\omega_i+1), \ \ \omega_i\in(0,\infty),\ \ i=1,2,3,
$$
$$
\rho_{ij}^2=\Psi_i+\Psi_j-2\Psi_i^{1/2}\Psi_j^{1/2}\cos(\varphi_i-\varphi_j)-(\omega_i-\omega_j)^2,\ i, j=1,2,3, \ i\ne j. 
$$

Since we assume that the motion takes place in a plane
parallel with the $xy$-plane on $\mathbb H_\kappa^2$, we can take 
$$
\omega_i=v>0,  \ \ i=1,2,3.
$$
The rest of the proof is identical with that for Theorem \ref{isos-re}, with one exception: if we denote 
$$
\mu=|\kappa|^{1/2}v>0,\ \ \delta=-|\kappa|v(2|\kappa|^{-1/2}+v),
$$
then $\gamma$ in the proof of Theorem 3 must be replaced by $\delta$ and the discussion of the inequalities related to the quadratic equation $\lambda(2-\lambda)=\gamma$ must be replaced by that
of the quadratic equation in $\mu$,
$$
-\mu(\mu+2)=\delta.
$$
But then, in the equation $\delta=F(s)$, the only solutions
occur when $F(s)=0/0$, since otherwise $\delta$ is negative
and $F(s)$ positive. Like in the case of the spheres $\mathbb S_\kappa^2$, those solutions correspond to the
Lagrangian (equilateral) triangles, a remark that completes the proof. 
\end{proof}

\section{Equivalent equations of motion}

In this section we will obtain another form of the equations
of motion that will be suitable for the study of 
Lagrangian relative equilibria. Let us notice first that system \eqref{new} is not analytic in $\kappa$ at $\kappa=0$ due to the occurrence of the terms $|\kappa|^{1/2}$ in the last equation. In \cite{Diacu77}, this inconvenience was solved by applying the substitution  
$
\delta=\sigma|\kappa|^{1/2},
$
and further using the parameter $\delta$ instead of $\kappa$. Here we will proceed differently. The idea, which Carles Sim\'o kindly suggested, is to express the variables $\omega_1, \omega_2, \omega_3$ in terms of the other variables in a suitable way with the help of the constraints, such that system \eqref{new} becomes analytic for all values of $\kappa\in\mathbb R$. For this, we will write the constraints \eqref{constraint-1} as 
$$
\kappa(x_i^2+y_i^2+z_i^2)+(|\kappa|^{1/2}\omega_i+1)^2=1,
\ \ i=1,2,3,
$$
which we solve for $\omega_i$ explicitly and obtain
\begin{equation}\label{omega}
\omega_i=|\kappa|^{-1/2}\Big[\sqrt{1-\kappa(x_i^2+y_i^2+z_i^2)}-1\Big], \ \ i=1,2,3.
\end{equation}
We can now completely eliminate the three equations involving $\omega_1,\omega_2,\omega_3$, but these variables still occur in the terms $r_{ij}^2$, which show up in the other equations. Actually these variables appear in the particular form $\sigma(\omega_i-\omega_j)^2,$
which using \eqref{omega} can be written as
\begin{equation}\label{omega-2}
\sigma(\omega_i-\omega_j)^2=
\frac{\kappa(x_i^2+y_i^2+z_i^2-x_j^2-y_j^2-z_j^2)^2}{\Big[\sqrt{1-\kappa(x_i^2+y_i^2+z_i^2)}+\sqrt{1-\kappa(x_j^2+y_j^2+z_j^2)}\ \Big]^2}.
\end{equation}
For $\kappa>0$ and sufficiently small, the expressions involving square roots always exist, an assumption we will further impose in the rest of the paper. So system \eqref{new} is now reduced to the first $18$ equations
without any constraints and is analytic in $\kappa$ for all small values of this parameter. In fact, as we mentioned earlier, from the physical point of view these values are the only interesting ones because should our universe be non-flat, it would for sure have a curvature that is close to zero, whether positive or negative.

 Using equations \eqref{new}, as well as formulas \eqref{omega} and \eqref{omega-2} with $z_1=z_2=z_3=0$, some differentiation leads us to the system 
\begin{equation}\label{3-body-problem}
\begin{cases}
\ddot x_i=\sum_{j=1, j\ne i}^3\cfrac{m_j\Big[x_j-\Big(1-\frac{\kappa \rho_{ij}^2}{2}\Big)x_i\Big]}{\rho_{ij}^3\Big(1-\frac{\kappa \rho_{ij}^2}{4}\Big)^{3/2}}-\kappa (\dot{x}_i^2+\dot{y}_i^2+\kappa B_i)x_i\cr
\ddot y_i=\sum_{j=1, j\ne i}^3\cfrac{m_j\Big[y_j-\Big(1-\frac{\kappa \rho_{ij}^2}{2}\Big)y_i\Big]}{\rho_{ij}^3\Big(1-\frac{\kappa \rho_{ij}^2}{4}\Big)^{3/2}}-\kappa (\dot{x}_i^2+\dot{y}_i^2+\kappa B_i)y_i,\cr
\end{cases}
\end{equation}
$i=1,2,3$, where, for $ i,j\in\{1,2,3\},\ i\ne j,$
\begin{equation}\label{rho}
\rho_{ij}^2=(x_i-x_j)^2+(y_i-y_j)^2+
\frac{\kappa(A_i-A_j)^2}{\big(\sqrt{1-\kappa A_i}+\sqrt{1-\kappa A_j}\ \! \big)^2},
\end{equation}
\begin{equation}\label{Ai}
A_i=x_i^2+y_i^2,\ \ i=1,2,3,
\end{equation}
\begin{equation}\label{Bi}
B_i=\frac{(x_i\dot{x}_i+y_i\dot{y}_i)^2}{1-\kappa A_i},\ \ i=1,2,3.
\end{equation}

This is the system we will further study here. Notice that for $\kappa=0$ we recover the classical Newtonian equations of the planar 3-body problem,
\begin{equation}\label{3-body-problem-Newtonian}
\begin{cases}
\ddot x_i=\sum_{j=1, j\ne i}^3\frac{m_j(x_j-x_i)}{\rho_{ij}^3}\cr
\ddot y_i=\sum_{j=1, j\ne i}^3\frac{m_j(y_j-y_i)}{\rho_{ij}^3}, \ \ i=1,2,3. \cr
\end{cases}
\end{equation}

For all $\kappa\in\mathbb R$, system \eqref{3-body-problem} possesses the integral of energy,
\begin{equation}
T_\kappa({\bf q},\dot{\bf q})-U_\kappa({\bf q})=h,
\end{equation}
where $h$ is an integration constant, $T_\kappa$ is the kinetic energy,
\begin{equation}
T_\kappa({\bf q},\dot{\bf q})=\frac{1}{2}\sum_{i=1}^3m_i(\dot{x}_i^2+\dot{y}_i^2+\kappa B_i),
\end{equation}
and $U_\kappa$ is the force function,
\begin{equation}
U_\kappa({\bf q})=\sum_{1\le i<j\le 3}\frac{m_im_j\Big(1-\frac{\kappa \rho_{ij}^2}{2}\Big)}{\rho_{ij}\Big(1-\frac{\kappa\rho_{ij}^2}{4}\Big)^{1/2}},
\end{equation}
with
$$
{\bf q}=({\bf q}_1,{\bf q}_2, {\bf q}_3), \ {\bf q}_i=(x_i, y_i),\ i=1,2,3.
$$
Notice that for $\kappa=0$ we recover the well-known expression of the kinetic energy,
$$
T({\bf q}, \dot{\bf q})=\frac{1}{2}\sum_{i=1}^3m_i(\dot{x}_i^2+\dot{y}_i^2),
$$
and the force function,
$$
U({\bf q})=\sum_{i=1}^3\frac{m_im_j}{\rho_{ij}}.
$$

To write the integrals of the total angular momentum, notice first that from \eqref{omega} we have that
$$
\dot\omega_i=-\sigma|\kappa|^{1/2}B_i^{1/2},\ \ i=1,2,3,
$$
and by multiplying with the conjugate in \eqref{omega}, we can write that
$$
\omega_i=-\frac{\sigma|\kappa|^{1/2}A_i}{1+\sqrt{1-\kappa A_i}},\ \ i=1,2,3.
$$
Using these expressions, we can now write the three integrals of the total angular momentum,
\begin{equation}\label{1-angmom}
\sigma|\kappa|^{1/2}\sum_{i=1}^3m_i\bigg(B_i^{1/2}x_i-\frac{A_i\dot{x}_i}{1+\sqrt{1-\kappa A_i}}\bigg)-
|\kappa|^{-1/2}\sum_{i=1}^3m_i\dot{x}_i=c_1,
\end{equation}
\begin{equation}\label{2-angmom}
\sigma|\kappa|^{1/2}\sum_{i=1}^3m_i\bigg(B_i^{1/2}{y}_i-\frac{A_i\dot{y}_i}{1+\sqrt{1-\kappa A_i}}\bigg)-|\kappa|^{-1/2}\sum_{i=1}^3m_i\dot{y}_i=c_2,
\end{equation}
\begin{equation}\label{3-angmom}
\sum_{i=1}^3\sigma m_i(y_i\dot{x}_i-x_i\dot{y}_i)=c_3,
\end{equation}
where $c_1, c_2, c_3$ are integration constants. Notice that if we multiply equations \eqref{1-angmom} and \eqref{2-angmom} by $|\kappa|^{1/2}$, then for $\kappa=0$ these two integrals of the angular momentum become the two integrals of the linear momentum, so we are left only with one integral of the total angular momentum, equation \eqref{3-angmom}, as expected to happen in the planar Euclidean case. 

So for $\kappa=0$, we can write the integrals of the centre of mass and linear momentum as
\begin{equation}
\begin{cases}
\sum_{i=1}^3m_ix_i=0,\cr
\sum_{i=1}^3m_iy_i=0,\cr
\sum_{i=1}^3m_i\dot{x}_i=0,\cr
\sum_{i=1}^3m_i\dot{y}_i=0.\cr
\end{cases}
\end{equation}
No such integrals, however, occur for $\kappa\ne 0$, as shown in \cite{Diacu77}.

\section{The case of two negligible masses}

In this section we consider the simple case
when $m_1=:m>0$ and $m_2=m_3=0$, also known as a planetary problem. For the classical Newtonian equations, the
Lagrangian relative equilibria have a particular form under such circumstances: $m_1$ is at rest at the origin of the coordinate system, while $m_2$ and $m_3$ move along the same circle, such that the three particles form an equilateral triangle for all time. We will further show that similar orbits exist for $\kappa\ne 0$. More precisely, we will prove the
following result.

\begin{theorem}
In the case of two negligible masses, there occur no bifurcations of the Lagrangian equilateral relative equilibria
when passing from $\mathbb S_\kappa^2$ to $\mathbb R^2$
to $\mathbb H_\kappa^2$, as $\kappa$ goes from $+\infty$ to $-\infty$.
\end{theorem}
\begin{proof}
Since $m_2=m_3=0$, these particles do not influence the motion of $m_1$. Therefore if we initially take $x_1(0)=y_1(0)=0$, the coordinates of $m_1$ remain the same all along the motion. Then the equations in \eqref{3-body-problem} corresponding to $x_1$ and $y_1$ are identically satisfied and system \eqref{3-body-problem}, which now describes only the motion of $m_2$ and $m_3$, takes the form
\begin{equation}\label{two-zero-masses}
\begin{cases}
\ddot{x}_2=\frac{m\big(\frac{\kappa\rho_{12}^2}{2}-1\big)x_2}{\rho_{12}^3\big(1-\frac{\kappa\rho_{12}^2}{4}\big)^{3/2}}
-\kappa(\dot{x}_2^2+\dot{y}_2^2+\kappa B_2)x_2\cr
\ddot{y}_2=\frac{m\big(\frac{\kappa\rho_{12}^2}{2}-1\big)y_2}{\rho_{12}^3\big(1-\frac{\kappa\rho_{12}^2}{4}\big)^{3/2}}
-\kappa(\dot{x}_2^2+\dot{y}_2^2+\kappa B_2)y_2\cr
\ddot{x}_3=\frac{m\big(\frac{\kappa\rho_{13}^2}{2}-1\big)x_3}{\rho_{13}^3\big(1-\frac{\kappa\rho_{13}^2}{4}\big)^{3/2}}
-\kappa(\dot{x}_3^2+\dot{y}_3^2+\kappa B_2)x_3\cr
\ddot{y}_3=\frac{m\big(\frac{\kappa\rho_{13}^2}{2}-1\big)y_3}{\rho_{13}^3\big(1-\frac{\kappa\rho_{13}^2}{4}\big)^{3/2}}
-\kappa(\dot{x}_3^2+\dot{y}_3^2+\kappa B_2)y_3,\cr
\end{cases}
\end{equation}
where
$$
\rho_{12}^2=x_2^2+y_2^2+\frac{\kappa(x_2^2+y_2^2)^2}{\big[1+\sqrt{1-\kappa(x_2^2+y_2^2)}\big]^2},
$$
$$
\rho_{13}^2=x_3^2+y_3^2+\frac{\kappa(x_3^2+y_3^2)^2}{\big[1+\sqrt{1-\kappa(x_3^2+y_3^2)}\big]^2},
$$
$$
B_2=\frac{(x_2\dot{x}_2+y_2\dot{y}_2)^2}{1-\kappa(x_2^2+y_2^2)},\ \ \  B_3=\frac{(x_3\dot{x}_3+y_3\dot{y}_3)^2}{1-\kappa(x_3^2+y_3^2)}.
$$

We can now show that system \eqref{two-zero-masses} has a solution of the form
\begin{equation}\label{solution-2-0-masses}
x_2=r\cos\alpha t,\ \ y_2=r\sin\alpha t,\ \ x_3=r\cos(\alpha t+\theta),\ \ y_3=r\sin(\alpha t+\theta),
\end{equation}
where $\theta\in(0,\pi)$. This means that $m_2$ and $m_3$ move along a circle of radius $r$ on the 2-sphere or the hyperbolic 2-sphere of curvature $\kappa$, such that they form an equilateral triangle with $m_1$ at all times. In the 3-dimensional flat ambient space, the particles form an equilateral triangle that  is not parallel with the $xy$-plane. The projection of the $2\pi/3$-angle between the sides $m_1m_2$ and $m_1m_3$ on the $xy$-plane is $\theta$,
an angle that depends on $r$. The projection of the other two angles of the equilateral triangle onto the $xy$-plane is also different from $2\pi/3$.

To show that \eqref{solution-2-0-masses} is a solution of system \eqref{two-zero-masses}, notice first that
$$
\rho^2:=\rho_{12}^2=\rho_{13}^2=\frac{2r^2}{1+\sqrt{1-\kappa r^2}},\ \ \ B_2=B_3=0.
$$
But the point masses $m_1, m_2, m_3$ form an equilateral triangle only if the condition 
$$
\rho^2=\rho_{23}^2=2(1-\cos\theta)r^2
$$ 
is also satisfied. This condition leads to the connection between $\theta$ and $r$, namely
$$
\cos\theta=\frac{\sqrt{1-\kappa r^2}}{1+\sqrt{1-\kappa r^2}}.
$$
Then some straightforward computations prove that all four equations in system \eqref{two-zero-masses} lead to the same relationship,
\begin{equation}\label{relationship}
\alpha^2=\frac{m(1-\kappa r^2+\sqrt{1-\kappa r^2})}{r^3(1-\kappa r^2)(1+\sqrt{1-\kappa r^2})},
\end{equation}
which shows how the angular velocity $\alpha$ of the particles $m_2$ and $m_3$ depends on the constants $\kappa, m$, and $r$. Since for $\kappa, m$, and $r$ fixed there are always two values of $\alpha$ that satisfy the above relationship, one corresponding to each direction of rotation, it means that Lagrangian relative equilibria of this kind exist.
\end{proof}

\section{The case of one negligible mass}

We further consider the case when $m_1=:M>0, m_2=:m>0$ and $m_3=0$. For the classical Newtonian equations, the
Lagrangian relative equilibria have a particular form under such circumstances: $m_1$ and $m_2$ lie on a straight line that rotates around the centre of mass of these particles, each moving along a circle, while $m_3$ moves on another circle, concentric with the other two, such that all three particles form an equilateral triangle at all times. We will further find out what happens for $\kappa\ne 0$.

Since $m_3=0$, this particle does not influence the motion of $m_1$ and $m_2$. System \eqref{3-body-problem} thus takes the form
\begin{equation}\label{1-negligible}
\begin{cases}
\ddot x_1=\frac{m\Big[x_2-\Big(1-\frac{\kappa \rho_{12}^2}{2}\Big)x_1\Big]}{\rho_{12}^3\Big(1-\frac{\kappa \rho_{12}^2}{4}\Big)^{3/2}}-\kappa (\dot{x}_1^2+\dot{y}_1^2+\kappa B_1)x_1\cr
\ddot y_1=\frac{m\Big[y_2-\Big(1-\frac{\kappa \rho_{12}^2}{2}\Big)y_1\Big]}{\rho_{12}^3\Big(1-\frac{\kappa \rho_{12}^2}{4}\Big)^{3/2}}-\kappa (\dot{x}_1^2+\dot{y}_1^2+\kappa B_1)y_1\cr
\ddot x_2=\frac{M\Big[x_1-\Big(1-\frac{\kappa \rho_{12}^2}{2}\Big)x_2\Big]}{\rho_{12}^3\Big(1-\frac{\kappa \rho_{12}^2}{4}\Big)^{3/2}}-\kappa (\dot{x}_2^2+\dot{y}_2^2+\kappa B_2)x_2\cr
\ddot y_2=\frac{M\Big[y_1-\Big(1-\frac{\kappa \rho_{12}^2}{2}\Big)y_2\Big]}{\rho_{12}^3\Big(1-\frac{\kappa \rho_{12}^2}{4}\Big)^{3/2}}-\kappa (\dot{x}_2^2+\dot{y}_2^2+\kappa B_2)y_2\cr
\ddot x_3=\frac{M\Big[x_1-\Big(1-\frac{\kappa \rho_{13}^2}{2}\Big)x_3\Big]}{\rho_{13}^3\Big(1-\frac{\kappa \rho_{13}^2}{4}\Big)^{3/2}}+
\frac{m\Big[x_2-\Big(1-\frac{\kappa \rho_{23}^2}{2}\Big)x_3\Big]}{\rho_{23}^3\Big(1-\frac{\kappa \rho_{23}^2}{4}\Big)^{3/2}}
-\kappa (\dot{x}_3^2+\dot{y}_3^2+\kappa B_3)x_3\cr
\ddot y_3=\frac{M\Big[y_1-\Big(1-\frac{\kappa \rho_{13}^2}{2}\Big)y_3\Big]}{\rho_{13}^3\Big(1-\frac{\kappa \rho_{13}^2}{4}\Big)^{3/2}}+
\frac{m\Big[y_2-\Big(1-\frac{\kappa \rho_{23}^2}{2}\Big)y_3\Big]}{\rho_{23}^3\Big(1-\frac{\kappa \rho_{23}^2}{4}\Big)^{3/2}}
-\kappa (\dot{x}_3^2+\dot{y}_3^2+\kappa B_3)y_3,\cr
\end{cases}
\end{equation}
which is decoupled, since the first four equations are independent of the last two equations. Here we have denoted by
$$
\rho_{12}^2=(x_1-x_2)^2+(y_1-y_2)^2+\frac{\kappa(x_1^2+y_1^2-x_2^2-y_2^2)^2}{\big[\sqrt{1-\kappa(x_1^2+y_1^2)}+\sqrt{1-\kappa(x_2^2+y_2^2)}\big]^2},
$$
$$
\rho_{13}^2=(x_1-x_3)^2+(y_1-y_3)^2+\frac{\kappa(x_1^2+y_1^2-x_3^2-y_3^2)^2}{\big[\sqrt{1-\kappa(x_1^2+y_1^2)}+\sqrt{1-\kappa(x_3^2+y_3^2)}\big]^2},
$$
$$
\rho_{23}^2=(x_2-x_3)^2+(y_2-y_3)^2+\frac{\kappa(x_2^2+y_2^2-x_3^2-y_3^2)^2}{\big[\sqrt{1-\kappa(x_2^2+y_2^2)}+\sqrt{1-\kappa(x_3^2+y_3^2)}\big]^2},
$$
$$
B_1=\frac{(x_1\dot{x}_1+y_1\dot{y}_1)^2}{1-\kappa(x_1^2+y_1^2)},\ \ \ B_2=\frac{(x_2\dot{x}_2+y_2\dot{y}_2)^2}{1-\kappa(x_2^2+y_2^2)},\ \ \  B_3=\frac{(x_3\dot{x}_3+y_3\dot{y}_3)^2}{1-\kappa(x_3^2+y_3^2)}.
$$

In the flat case, $m_1$ and $m_2$ lie on an
axis that rotates around the centre of mass of these particles. Thus $m_1$ and $m_2$ move, in general, on concentric circles; if $m_1=m_2$, they move on the same circle. The particle $m_3$ forms all the time an equilateral triangle with $m_1$ and $m_2$ and moves on a circle that is concentric with the other (one or two) circles. 

We will therefore place $m_1$ and $m_2$ on a geodesic passing through and rotating around the contact point of $\mathbb S_\kappa^2, \mathbb R^2$, and $\mathbb H_\kappa^2$ (see Figure \ref{common}). Then, if a Lagrangian orbit exists, $m_3$ must move on a circle of $\mathbb S_\kappa^2$ or $\mathbb H_\kappa^2$. So if we take a fixed angle $\theta\in(0,\pi)$, we are seeking a solution of the form
\begin{align}
x_1&=r_1\cos\alpha t, \ \ \ \ \ \ \ \ \ \   y_1=r_1\sin\alpha t\label{x1y1-1negl} \\
x_2&=-r_2\cos\alpha t,\ \ \ \ \ \ \ \ \! y_2=-r_2\sin\alpha t
\label{x2y2-1negl} \\
x_3&=r_3\cos(\alpha t+\theta), \ \ \ y_3=r_3\sin(\alpha t+\theta), \label{x3y3-1negl}
\end{align}
with $r_1, r_2, r_3>0$ and, in the case of $\mathbb S_\kappa^2$, no larger than the radius of the sphere. Then it follows that for this candidate solution we have
\begin{equation}
B_1=B_2=B_3=0,
\end{equation}
\begin{equation}\label{rho12}
\rho_{12}^2=(r_1+r_2)^2+\frac{\kappa(r_1^2-r_2^2)^2}{(\sqrt{1-\kappa r_1^2}+\sqrt{1-\kappa r_2^2})^2},
\end{equation}
\begin{equation}\label{rho13}
\rho_{13}^2=r_1^2+r_2^2+2r_1r_2\cos\theta+\frac{\kappa(r_1^2-r_3^2)^2}{(\sqrt{1-\kappa r_1^2}+\sqrt{1-\kappa r_3^2})^2},
\end{equation}
\begin{equation}\label{rho23}
\rho_{23}^2=r_2^2+r_3^2+2r_2r_3\cos\theta+\frac{\kappa(r_2^2-r_3^2)^2}{(\sqrt{1-\kappa r_2^2}+\sqrt{1-\kappa r_3^2})^2}.
\end{equation}
Since the triangle is equilateral, we must have
\begin{equation}\label{the-rhos}
\rho_{12}=\rho_{13}=\rho_{23}=:\rho.
\end{equation}

Substituting \eqref{x1y1-1negl} and \eqref{x2y2-1negl} in  the first and third as well as in the second and fourth equations of system \eqref{1-negligible}, we can respectively conclude that
\begin{equation}\label{alpha-12-1negl}
\alpha^2r_1(1-\kappa r_1^2)=\frac{m\big(r_1+r_2-\frac{\kappa\rho^2 r_1}{2}\big)}{\rho^3\big(1-\frac{\kappa\rho^2}{4}\big)^{3/2}},
\end{equation}
\begin{equation}\label{alpha-21-1negl}
\alpha^2r_2(1-\kappa r_2^2)=\frac{M\big(r_1+r_2-\frac{\kappa\rho^2 r_2}{2}\big)}{\rho^3\big(1-\frac{\kappa\rho^2}{4}\big)^{3/2}},
\end{equation}
relationships which imply that the Lagrangian solution of the decoupled system given by the first four equations in \eqref{1-negligible} exists only if
\begin{equation}\label{eq-condition}
\frac{m\big(r_1+r_2-\frac{\kappa\rho^2 r_1}{2}\big)}{r_1(1-\kappa r_1^2)}=\frac{M\big(r_1+r_2-\frac{\kappa\rho^2 r_2}{2}\big)}{r_2(1-\kappa r_2^2)}.
\end{equation}
This identity is obviously satisfied when $r_1=r_2$ and $M=m$. Let us first deal with this case.

\subsection{Two equal masses and one negligible mass in $\mathbb S_\kappa^2$}
\label{sphere-2eq-1neg}

In this subsection we will prove the following result.

\begin{theorem}\label{strange}
Assume that two bodies of equal mass move on a non-equatorial parallel circle of the sphere $\mathbb S_\kappa^2$,
being always diametrically opposed and form a Lagrangian  relative equilibrium with a third body, which has negligible mass. Then the circle on which the two bodies move must
have its radius equal to $(2\kappa)^{-1/2}$ and the third body
must move on the equator. In other words, given $\kappa$, the size of the equilateral triangle does not depend on the value of the equal masses, but the angular velocity of the equilateral triangle does. 
\end{theorem}
\begin{proof}
So we assume that the motion takes place in $\mathbb S_\kappa^2$, that $0<r_1=r_2=:r<\kappa^{-1/2}$, $M=m$, and $m_3$ is negligible. Notice that $\kappa^{-1/2}=:R_\kappa$ is the radius of $\mathbb S_\kappa^2$. Also remark that in this case the geodesic passing through $m_3$ and the North Pole of the sphere is orthogonal to the geodesic connecting $m_1$ and $m_2$. If we project these geodesics on the $xy$-plane, the projections are also orthogonal. This implies that the angle $\theta$ taken in \eqref{x3y3-1negl} is $\pi/2$, so the solution we are now checking is of the form
\begin{align}
x_1&=r\cos\alpha t, \ \ \ \ \ \ \ \!   y_1=r\sin\alpha t\label{em-x1y1-1negl} \\
x_2&=-r\cos\alpha t,\ \ \ \   y_2=-r\sin\alpha t
\label{em-x2y2-1negl} \\
x_3&=-r_3\sin\alpha t, \ \ \ y_3=r_3\cos\alpha t. \label{em-x3y3-1negl}
\end{align}

From \eqref{rho12} and \eqref{the-rhos}, we can conclude that for a solution of system \eqref{1-negligible} of the form \eqref{em-x1y1-1negl}--\eqref{em-x3y3-1negl} we have
$$
\rho=2r.
$$
This fact together with relations \eqref{rho13} and \eqref{rho23} for $\theta=\pi/2$ and with equation \eqref{the-rhos} imply on one hand that
\begin{equation}\label{constraint-from-rhos}
3r^2=r_3^2+\frac{\kappa(r^2-r_3^2)^2}{(\sqrt{1-\kappa r^2}+\sqrt{1-\kappa r_3^2})^2}.
\end{equation}
On the other hand, we have from \eqref{alpha-12-1negl} and
\eqref{alpha-21-1negl} that
\begin{equation}\label{one-alpha}
\alpha^2(1-\kappa r^2)=\frac{m\big(2-\frac{\kappa\rho^2}{2}\big)}{\rho^3\big(1-\frac{\kappa\rho^2}{4}\big)^{3/2}},
\end{equation}
and from the last two equations in \eqref{1-negligible} that
\begin{equation}\label{two-alpha}
\alpha^2(1-\kappa r_3^2)=\frac{2m\big(1-\frac{\kappa\rho^2}{2}\big)}{\rho^3\big(1-\frac{\kappa\rho^2}{4}\big)^{3/2}}.
\end{equation}
By comparing \eqref{one-alpha} and \eqref{two-alpha} we are led to the relationship
\begin{equation}\label{r3-r}
r_3=\sqrt{2}r.
\end{equation}
Substituting \eqref{r3-r} into \eqref{constraint-from-rhos}, we obtain 
\begin{equation}
r=(2\kappa)^{-1/2}\ \ {\rm or}\ \ r=R_\kappa/\sqrt{2},
\end{equation}
which also implies that $r_3=R_\kappa$. Then we also obtain that
\begin{equation}
\alpha^2=2m\kappa^{3/2}\ \ {\rm or}\ \ \alpha^2=2m/R_\kappa^3.
\end{equation}

From the geometric-dynamical point of view, the above results describe a surprising situation. Since $r_3=R_\kappa$, it means that the projection of the height from $m_3$ of the equilateral triangle $m_1m_2m_3$ onto the $xy$-plane has the same Euclidean length as the radius of the sphere. Therefore the particle $m_3$ must rotate on the equator that is parallel to the $xy$-plane, while the particles $m_1$ and $m_2$ rotate on a non-geodesic circle that is also parallel with the $xy$-plane. This circle has radius $r=R_\kappa/\sqrt{2}$, so it is uniquely determined only by the given sphere (therefore cannot be chosen arbitrarily). So, for a given sphere, the equilateral triangle has always the same size, independently of the values of the masses. But the angular velocity, $\alpha$, depends on the value, $m$, of the two equal masses and on the curvature, $\kappa>0$, or radius $R_\kappa$, of the sphere $\mathbb S_\kappa^2$. 
This remark completes the proof.
\end{proof}

\subsection{Two equal masses and one negligible mass in $\mathbb H_\kappa^2$}
\label{hyperbolic-sphere-2eq-1neg}

In this subsection we will prove the following result, which shows again that the dynamics on $\mathbb H_\kappa^2$
is not as rich as in $\mathbb S_\kappa^2$.
\begin{theorem}
There are no Lagrangian relative equilibria in $\mathbb H_\kappa^2$ with two bodies of equal mass and a third body of negligible mass.
\end{theorem}
\begin{proof}
So assume that the motion takes place in $\mathbb H_\kappa^2$, that $r_1=r_2=:r>0$, $M=m$, and $m_3$ is negligible. In this case the imaginary radius of the hyperbolic sphere is $(-\kappa)^{-1/2}=:R_\kappa$.
We can proceed as in the case of the sphere discussed in Subsection \ref{sphere-2eq-1neg}, and the solution we are checking has the same form, 
\eqref{em-x1y1-1negl}--\eqref{em-x3y3-1negl}. The computations are then identical up to formula \eqref{r3-r}. But after substituting \eqref{r3-r} into \eqref{constraint-from-rhos}, we are led to the conclusion that
$$
\frac{\kappa r^2}{(\sqrt{1-\kappa r^2}+\sqrt{1-2\kappa r^2})^2}=1,
$$
which is impossible since $\kappa<0$. Therefore we must conclude that there are no Lagrangian relative equilibria of this type in $\mathbb H_\kappa^2$.
\end{proof}

\subsection{The general case}

We now return to the general case and assume, without loss of generality, that $M\ge m$. Under this hypothesis we can prove the following result.

\begin{theorem}
If one of the three masses is negligible,

(i)  there are no Lagrangian relative equilibria in $\mathbb H_\kappa^2$;

(ii) there are no Lagrangian relative equilibria in $\mathbb S_\kappa^2$ if the curvature $\kappa$ is sufficiently small, unless the two non-negligible masses are equal, in which case the solutions occur as stated in Theorem \ref{strange}.
\end{theorem}
\begin{proof}
If we substitute a candidate solution of the form \eqref{x1y1-1negl}--\eqref{x3y3-1negl} into the last two equations of system
\eqref{1-negligible}, we obtain equations that involve
$\cos\alpha t, \sin\alpha t$, on one hand, and $\cos(\alpha t+\theta), \sin(\alpha t+\theta)$ on the other hand. Separating the arguments $\alpha t$ and $\theta\in(0,\pi)$, and arguing that the equations must be satisfied for all $t\in\mathbb R$, we are led to the
relationships
\begin{equation}\label{sin-alphat}
\alpha^2r_3(1-\kappa r_3^2)=\frac{(M+m)\big(1-\frac{\kappa\rho^2}{2}\big)r_3}{\rho^3\big(1-\frac{\kappa\rho^2}{4}\big)^{3/2}},
\end{equation}
\begin{equation}\label{cos-alphat}
\alpha^2r_3(1-\kappa r_3^2)\cos\theta=-\frac{Mr_1-mr_2}{\rho^3\big(1-\frac{\kappa\rho^2}{4}\big)^{3/2}}+\frac{(M+m)\big(1-\frac{\kappa\rho^2}{2}\big)r_3\cos\theta}{\rho^3\big(1-\frac{\kappa\rho^2}{4}\big)^{3/2}},
\end{equation}
which are simultaneously satisfied only if 
\begin{equation}\label{Mr1=mr2}
Mr_1=mr_2.
\end{equation}
The case of $M=m$ and $r_1=r_2$, which corresponds to $\theta=\pi/2$ and was already treated in subsections \ref{sphere-2eq-1neg} and \ref{hyperbolic-sphere-2eq-1neg},
is in agreement with relation \eqref{Mr1=mr2} and its derivation from \eqref{sin-alphat} and \eqref{cos-alphat}.

Writing \eqref{Mr1=mr2} as $m/r_1=M/r_2$, we can conclude from \eqref{eq-condition} that
\begin{equation}\label{factorization}
(r_1-r_2)\Big[\frac{\rho^2}{2}(1+\kappa r_1r_2)-(r_1+r_2)^2\Big]=0.
\end{equation}
Therefore we must split our analysis into two cases.

\bigskip

{\bf Case 1}: $r_1=r_2$. With this hypothesis it follows from \eqref{Mr1=mr2} that $M=m$, a situation we already settled in subsections \ref{sphere-2eq-1neg} and \ref{hyperbolic-sphere-2eq-1neg}. Moreover, for $\kappa>0$, the relationship between $r$ and $r_3$ that we obtain from \eqref{eq-condition} and \eqref{sin-alphat} is identical with the one we computed in subsection \ref{sphere-2eq-1neg}, so this case is completely solved.

\bigskip

{\bf Case 2}: $\frac{\rho^2}{2}(1+\kappa r_1r_2)-(r_1+r_2)^2=0$. Under these circumstances we must start with the hypothesis $r_1\ne r_2$. Since we assumed $M\ge m$, relation \eqref{Mr1=mr2} implies that we cannot have $r_1>r_2$, so the initial hypothesis leads to the conclusion that $r_1<r_2$ and, consequently, $M>m$. If we denote $r:=r_1$ and $\mu:=M/m\ge 1$, then $r_2=\mu r$ and
\begin{equation}\label{r1r2as-r}
\frac{\rho^2}{2}(1+\kappa r_1r_2)-(r_1+r_2)^2=
\frac{\rho^2}{2}(1+\kappa\mu r^2)-(1+\mu)^2r^2.
\end{equation}
But from \eqref{rho12}, we obtain that
\begin{equation}\label{rho^2}
\rho^2=(1+\mu)^2r^2+\frac{\kappa(1-\mu)^2r^4}{(\sqrt{1-\kappa r^2}+\sqrt{1-\kappa\mu^2 r^2})^2}.
\end{equation}

If $\kappa<0$, we can see from \eqref{rho^2} that
$$
\rho^2\le (1+\mu)^2r^2.
$$
Consequently, using the above relation and \eqref{r1r2as-r}, we have the inequality 
$$
\frac{\rho^2}{2}(1+\kappa r_1r_2)-(r_1+r_2)^2\le
\frac{1}{2}(1+\mu)^2r^2(\kappa\mu r^2-1).
$$
But for $\kappa<0$, the right hand side of this inequality is negative, so the second factor in \eqref{factorization} cannot be zero. We can thus conclude from here and from the result obtained in subsection \ref{hyperbolic-sphere-2eq-1neg} that there are no Lagrangian relative equilibria on 2-dimensional hyperbolic spheres in the case of one negligible mass.

If $\kappa>0$ and sufficiently small, then 
$$
(\sqrt{1-\kappa r^2}+\sqrt{1-\kappa\mu^2 r^2})^2< 1.
$$
From \eqref{rho^2} we can thus draw the conclusion that
$$
\rho^2< (1+\mu)^2r^2+\kappa(1-\mu)^2r^4.
$$
Using the above relationship and \eqref{r1r2as-r} we obtain the inequality
\begin{equation}\label{essential-ineq}
\frac{\rho^2}{2}(1+\kappa r_1r_2)-(r_1+r_2)^2< 
\frac{1}{2}[(1+\mu)^2r^2(\kappa\mu r^2-1)+\kappa(1-\mu)^2r^4(1+\kappa\mu r^2)].
\end{equation}

Let us further see
\begin{equation}\label{ineq-mu^2}
\kappa\mu r^2-1\le 0
\end{equation}
even for large values of $\kappa>0$. For this, notice that from the second square root in the denominator of \eqref{rho^2}, we must have
$$
1-\kappa\mu^2r^2\ge 0,
$$
which is the same as 
$$
r\le\frac{m}{M}R_\kappa,
$$
where, recall, $\kappa=1/R_\kappa^2$. But $0<m/M<1$,
so the above inequality implies that
$$
r\le\sqrt{\frac{m}{M}}R_\kappa,
$$
a relationship that is equivalent to \eqref{ineq-mu^2} and that
is always satisfied.

From \eqref{ineq-mu^2} we now obtain that for $\kappa>0$ and sufficiently small, the right hand side of \eqref{essential-ineq} is negative, thus 
$$
\frac{\rho^2}{2}(1+\kappa r_1r_2)-(r_1+r_2)^2<0,
$$
and consequently there are no Lagrangian relative equilibria in Case 2.

So for one negligible mass, we can conclude that Lagrangian relative equilibria do not exist for negative curvature, and that they occur for small positive curvature only if $M=m$. Since the case $|\kappa|<\! \! <1$ is the only one of relevance when studying the curvature of the large-scale universe, we will not further investigate here what happens for other positive values of $\kappa$.
\end{proof}

\smallskip

\noindent{\bf Acknowledgments}. The author is indebted to Regina Mart\'inez, Carles Sim\'o, and Ernesto P\'erez-Chavela for several discussions on this topic.


\end{document}